\documentclass[a4paper,10pt,notitlepage]{article}

\usepackage{amssymb,amsmath,amsfonts,wasysym,hyperref,nameref,graphicx,bbm,ntheorem}

\newtheorem{Theorem}{Theorem}

\newtheorem{Proposition}{Proposition}
\newtheorem{Lemma}{Lemma}

\newenvironment{proof}[1][] {\noindent{\bf Proof#1.} }{\hspace*{\fill}$\square$\medskip\par}

\theorembodyfont{\upshape}

\newtheorem{Definition}{Definition}
\newtheorem{Remark}{Remark}

\newtheorem{Notation}{Notation}

\newcommand{\R}{\mathbb{R}}
\newcommand{\N}{\mathbb{N}}
\newcommand{\iN}{\in\mathbb{N}}
\newcommand{\Z}{\mathbb{Z}}
\newcommand{\E}{\mathbb{E}}
\newcommand{\IP}{\mathbb{P}}

\newcommand{\inftwo}[2]{\inf_{\substack{#1 \\ #2}}} 

\newcommand{\cd}{\mathfrak{d}}
\newcommand{\cT}{\mathcal{T}}
\newcommand{\cP}{\mathcal{P}}
\newcommand{\cK}{\mathcal{K}}
\newcommand{\cM}{\mathcal{M}}

\newcommand{\cU}{\mathcal{U}}
\newcommand{\cO}{\mathcal{O}}
\newcommand{\om}{\omega}
\newcommand{\Om}{\Omega}
\newcommand{\ind}{1\hspace{-0.098cm}\mathrm{l}}
\newcommand{\dd}{\mathrm{d}}
\newcommand{\eps}{\varepsilon}
\newcommand{\vphi}{\varphi}

\title{Constructive quantization: approximation by empirical measures} 
\author{Steffen Dereich\footnote{Philipps-Universit\"at Marburg, Fachbereich Mathematik und Informatik, Hans-Meerwein-Str., D-35032 Marburg, \{dereich,schottstedt\}@mathematik.uni-marburg.de} , Michael Scheutzow\footnote{Technische Universit\"at Berlin, Institut f\"ur Mathematik, 
Stra\ss e des 17. Juni 136, D-10623 Berlin, ms@math.tu-berlin.de}, and Reik Schottstedt\footnotemark[1]} 
\date{August 26, 2011}

\begin{document}

\maketitle

\begin{abstract}
In this article, we study the approximation of a probability measure $\mu$ on $\mathbb{R}^{d}$  by its empirical measure 
$\hat{\mu}_{N}$ interpreted as a random quantization. As error criterion we consider an averaged $p$-th moment  Wasserstein metric. 
In the case where $2p<d$, we establish refined upper and lower bounds for the error, a  \emph{high-resolution formula}. Moreover, 
we provide a universal estimate based on moments, a so-called \emph{Pierce type estimate}.
In particular, 
we show that  quantization by empirical measures is of optimal order under weak assumptions.
\end{abstract}
\medskip

\bigskip


\section{Introduction}\label{sec1}

Constructive quantization is concerned with the efficient computation of discrete approximations to probability distributions. 
The need for such approximations mainly stems  from two applications: firstly from information theory, where the approximation 
is a discretized version of an original signal which is to be stored on a data storage medium or transmitted via a channel 
(see e.g.\ \cite{Zad66,BuWi82, GG92}); secondly, from numerical integration, where integrals with respect to the original 
measure are replaced by the 
integral with respect to the discrete approximation (see e.g.\ \cite{PPP03}).



In both applications the objective is to find an optimal discrete subset of a metric space $(E,d)$ of cardinality $N$ say, a so-called \emph{codebook}, depending on the 
given probability measure $\mu$ on  $E$. 
 In the first application one further needs fast coding and  decoding schemes that  find for a signal a digital representation of a close element of the codebook or, resp., translate the digital representation back.
Clearly, the best coding scheme would map a signal to a digital representation of a closest neighbour in the codebook. The \emph{quantization number} measures the smallest possible averaged distance 
of a $\mu$-distributed  point to the codebook and hence the performance of the best possible approximate coding of $\mu$  using $N$ approximating points which corresponds to using $\log_2 N$ bits.

During the last decade, quantization attracted much interest mainly due to the second application, see for instance  \cite{PW11} for a recent review on financial applications. 
Here one aims at finding a codebook together with probability weights and the objective is to determine these in such a way that the 
distance between $\mu$ and the discrete probability measure is minimal with respect to some metric (e.g.\ a Wasserstein metric). Typically, the optimal solution of both problems are closely related. The optimal codebook of the first problem is also optimal for the second one and the optimal probability weights are the $\mu$-weights of the corresponding Voronoi cells. 
In particular, the optimal approximation errors are again the  \emph{quantization numbers}. 
A regularly updated list of articles dealing with quantization can be found at \verb?http://www.quantize.maths-fi.com/?.

From a constructive point of view, the two applications differ significantly and our research  is mainly motivated by the second application. For moderate codebook sizes and particular 
probability measures it is  feasible to run optimization algorithms and find approximations that are arbitrarily close to the 
optimum (see e.g.\ \cite{Pag98, PagPrin03}). See also \cite{MRY11} for a recent  constructive approach towards discrete approximation of marginals of stochastic differential equations. 
For large codebook sizes and probability measures that are defined 
implicitly, it is  often not feasible to find close to optimal quantizations in  reasonable time. For instance, large codebooks are necessary when using quantizations for  approximate sampling.



As an alternative approach we analyze the use of the \emph{empirical measure}~$\hat \mu_N$ 
generated by $N$ independent random variables distributed according to the original measure $\mu$. 
As error criterion we consider an averaged $L^p$-Wasserstein metric.
We stress that  in our case the codebook is generated by i.i.d.\ samples and that the weights all have \emph{equal} mass so that 
once the codebook is generated no further processing is needed. The advantage of using the empirical measure as a discrete approximation of $\mu$
is that it is usually easy to generate efficiently even for large $N$. The disadvantage is, of course, that for given $N$, the averaged Wasserstein distance 
between  $\mu$ and $\hat \mu_N$ is larger than that between $\mu$ and the optimal probability measure supported on $N$ points. We will show that in the case 
$E=\R^d$ equipped with some norm (which is the only case we consider in this article), the loss of performance is a multiplicative constant. 
While the empirical measure turns out to be a reasonable approximation that can be computed efficiently, the analysis of its performance is complicated by the fact that 
the problem is \emph{nonlocal} due to the fact that we take equal weights rather than optimal weights as in  \cite{Coh04,Yuk08} (see the following subsection).

A full treatment of  quantization typically  includes the derivation of asymptotic formulas in terms of the density of the absolutely continuous part of $\mu$, a \emph{high resolution formula}. 
Such a formula has been established for optimal quantization under norm-based distortions \cite{DGLP04}, for 
general Orlicz-norm distortions \cite{DerVor11}, and, very recently, also in the dual quantization problem \cite{PagWil10}. 
In this article, we prove a high resolution formula for the empirical measure under an averaged $L^p$-Wasserstein metric. 
Further, a Pierce type result is derived. In particular, we obtain order optimality of the new approach  under weak assumptions.

The article is organised as follows. Section~\ref{sec1}  introduces the basic notation and summarizes the main results. 
Section~\ref{sec_pierce} is devoted to the Pierce type result, see Theorem~\ref{thm_pierce} below. Section~\ref{sec_uni} treats 
the particular case where $\mu$ is the uniform distribution on $[0,1)^d$. It includes a proof of part (i) of Theorem~\ref{thm2} below. 
Finally, the high resolution formula provided by Theorem~\ref{thm2}  is proved in Section~\ref{sec_hrf}.  


\subsection{Notation}

We introduce the relevant notation along an example.
Consider the following problem arising from logistics. There is a demand for a certain economic good on $\mathbb{R}^{2}$ 
modelled by a finite measure $\mu$. The demand shall be accomodated by $N$ service centers that are placed at positions 
$x_{1},\dots,x_{N}\in\mathbb{R}^{2}$ and that have nonnegative capacities $p_1,\dots,p_N$ summing up to $\|\mu\|:=\mu(\R^2)$. 
We associate a given choice of \emph{supporting points} $x_1,\dots,x_N$ and \emph{weights} $p_1,\dots p_N$ with a measure 
$\hat \mu=\sum_{i=1}^N p_i \delta_{x_i}$, where $\delta_x$ denotes the Dirac measure in $x$. 
In order to cover the demand, goods have to be transported from the centers to the customers and we describe a transport 
schedule  by a measure $\xi$ on $\R^2\times \R^2$ such that its first, respectively second, marginal measure is equal to 
$\mu$, respectively $\hat \mu$. The set of admissible  transport schedules (\emph{transports}) is denoted by $\cM(\mu,\hat \mu)$  
and supposing that transporting a unit mass from $y$ to $x$ causes  cost $c(x,y)$, a transport $\xi\in\cM(\mu,\hat \mu)$ 
causes overall cost
$$
\int_{\R^2\times \R^2} c(x,y)\,\dd \xi(x,y).
$$

In  this article, we focus on norm based cost functions. In general, we assume that the demand is a finite measure on $\R^d$ 
and that the cost is of the form
$$
c(x,y)=\|x-y\|^p,
$$
where $p\ge 1$ and $\|\cdot\|$ is a fixed norm on $\R^d$. Given $\mu$ and $\hat \mu$, the minimal cost  is  the $p$th Wasserstein 
metric.

\begin{Definition}[$p$th Wasserstein metric]\label{pth Wasserstein metric}
 The $p$th {\em Wasserstein metric} of two finite measures $\mu$ and $\nu$  on
$\left(\mathbb{R}^{d},\mathcal{B}(\mathbb{R}^{d}) \right)$, which have equal mass, is given by
\[
 \rho_{p}(\mu,\nu)=\inf_{\xi\in \mathcal{M}(\mu,\nu)}\left(\int_{\mathbb{R}^{d}\times\mathbb{R}^{d}} 
\left\| x-y\right\|^{p}\xi\left( \dd x,\dd y\right)\right)^{1/p}
\]
where $\mathcal{M}(\mu,\nu)$ is the set of all finite measures $\rho$ on $\mathbb{R}^{d}\times\mathbb{R}^{d}$ having marginal 
distributions $\mu$ in the first component and $\nu$ in the second component.
\end{Definition}

The Wasserstein metric originates from the \textit{Monge-Kantorovich mass transportation problem}, which was introduced 
by G. Monge in 1781 \cite{Mon81}. Important results about the Wasserstein metric were achieved within the scope of 
\textit{transportation theory}, for instance by Kantorovich \cite{Kan42}, Kantorovich and Rubinstein \cite{KR58},  Wasserstein  
\cite{Was69}, Rachev and  R\"uschendorf \cite{RR98i,RR98ii} and others. 

Note that the Wasserstein metric is homogeneous in $(\mu,\nu)$ so that one can restrict attention to  probability measures. 
In this article, we analyse for a given probability measure $\mu$ on $\R^d$ the quality of the empirical measure as approximation. 
More explicitly, we denote by  $\hat\mu_N$  the (random) empirical measure of $N$ independent $\mu$-distributed random 
variables $X_1,\dots,X_N$, that is 
$$
\hat\mu_N=\frac 1N \sum_{j=1}^N \delta_{X_j},
$$
and, for fixed $p\ge 1$, we analyse the asymptotic behaviour of the so called \emph{random quantization error}
$$
V_{N,p}^\mathrm{rand}(\mu):=\E[\rho_p^p(\mu,\hat\mu_N)]^{1/p},
$$
as $N\in\N$ tends to infinity.

This quantity should be compared with the optimal approximation in the $L^p$-Wasserstein metric supported by $N$ points, that is
\begin{equation}\label{opt}
V_{N,p}^\mathrm{opt}(\mu):= \inf_{\nu} \rho_p(\mu,\nu),
\end{equation}
where the infimum is taken over all probability measures $\nu$ on $\R^d$ that are supported on $N$ points. 
The quantity $V_{N,p}^\mathrm{opt}(\mu)$ is local in the sense that for a given set $\mathcal C\subset \R^d$ of supporting points 
used in an approximation $\nu$, the optimal choice for $\nu$ is $\mu\circ \pi_\mathcal{C}^{-1}$, where $\pi_{\mathcal C}$ denotes a 
projection from $\R^d$ to $\mathcal C$. Hence, the minimisation of the latter quantity reduces to a minimisation over all sets 
$\mathcal C\subset \R^d$ of at most $N$ elements. Furthermore, the minimal error is the so called  $N$th \emph{quantization number}
$$
V_{N,p}^\mathrm{opt}(\mu)= \inf _{\mathcal C} \Bigl(\int  \min_{y\in \mathcal C} \|x-y\|^p \, \mu(\dd x)\Bigr)^{1/p}.
$$

For a measure $\mu$ on $\R^d$ we denote by $\mu=\mu_a+\mu_s$ its 
Lebesgue decomposition with $\mu_a$  denoting the absolutely continuous part with respect to Lebesgue measure $\lambda^d$ 
and $\mu_s$ the singular part. 

Further, we denote the uniform distibution on $[0,1)^d$ by $\cU$ and define
$$
\underline{V}_{N,p}^\mathrm{rand}:=\E\Bigl[ \inftwo{\cU' \in \Lambda}{\xi \in \mathcal{M}(\cU',\hat \cU_N)}\int_{\mathbb{R}^{d}\times\mathbb{R}^{d}} 
\left\| x-y\right\|^{p}\xi\left( \dd x,\dd y\right)\Bigr]^{1/p},
$$
where $\Lambda$ denotes the set of all probability measures $\cU'$ on $[0,1]^d$ which satisfy $\cU' (A) \le \mathcal U(A)$ for each Borel set 
$A \subset (0,1)^d$. Note that the latter quantity allows to have leakage in the boundaries of the support of the uniform measure~$\cU$. Therefore,  $\underline{V}_{N,p}^\mathrm{rand} \le V_{N,p}^\mathrm{rand}(\cU)$. It seems plausible that the ratio of
$\underline{V}_{N,p}^\mathrm{rand}$ and $V_{N,p}^\mathrm{rand}(\cU)$ converges to one as $N \to \infty$. However, this has not been proved yet.

\subsection{Main results}

We will assume throughout the paper that $d\ge 3$. The approximation by empirical measures satisfies a so-called 
{\em Pierce type} estimate.

\begin{Theorem}\label{thm_pierce}
Let $p\in\left[1,\frac{d}{2}\right)$ and $q > \frac{dp}{d-{p}}$. There exists a constant $\kappa_{p,q}^\mathrm{Pierce}$ such 
that for any probability measure $\mu$ on $\R^d$
\begin{align}\label{eq2111-1}
V_{N,p}^{\mathrm{rand}}(\mu) \le \kappa_{p,q}^{\mathrm{Pierce}} \, \left[ \int_{\R^d} \left\|x\right\|^q \,\dd \mu(x) \right]^{1/q} \, N^{-1/d}
\end{align}
for all $N\in\N$. 
\end{Theorem}

\begin{Remark}\begin{itemize}
\item The constant in the statement of Theorem \ref{thm_pierce} is explicit, see Theorem \ref{Pierce}. Its value depends on the
chosen norm on $\R^d$. 
\item For $p>\frac d2$ and discrete measures $\mu$, the random approach typically  induces errors $V_{N,p}^\mathrm{rand}(\mu)$ 
that are not of order $\mathcal{O}(N^{-1/d})$: take, for instance, two different points $a,b\in\R^d$ and let 
$\mu=\frac 12\delta_a+\frac12\delta_b$. Then $N\,\hat\mu_N(\{a\})$ is binomially distributed with parameters $N$ and $\frac12$. 
Consequently,
$$
V_{N,p}^\mathrm{rand}(\mu)=\E [\rho_p^p(\mu,\hat \mu_N)]^{1/p}= \|a-b\| \,\E\bigl [\bigl|\hat\mu_N(\{a\})-\frac12\bigr|\bigr]^{1/p}
$$
is of order $N^{-1/2p}$ and, hence, converges to zero strictly slower than $N^{-1/d}$.
\item In \cite{AKT83}, the case where $d=2$, $p=1$ and $\mu=\cU$ is treated. There it is found that  the $L^1$-Wasserstein distance between two independent realisations of $\hat\cU_N$ is typically of order $N^{-1/2} (\log N)^{1/2}$ which shows the necessity of the assumption $d\geq3$.
\item For the uniform distribution $\cU$ on $[0,1)^d$, the results of Talagrand \cite{Tal94} imply that 
$V_{N,p}^\mathrm{rand}(\cU)$ is always of order $N^{-1/d}$ as long as $d\ge 3$. 
\end{itemize} 
\end{Remark}

The following theorem is a \emph{high resolution formula} for quantization by empirical measures. 

\begin{Theorem}\label{thm2} Let $p\in[1,\frac {d}2)$. 
\begin{itemize}\item[(i)] Let $\mathcal U$ denote the uniform distribution on $[0,1)^d$. There exists a constant 
$\kappa^\mathrm{unif}_p\in(0,\infty)$ such that
$$
\lim_{N\to\infty} N^{1/d} \, V_{N,p}^\mathrm{rand}(\cU)=\kappa_p^\mathrm{unif}.
$$
Further, there exist a constant $\underline{\kappa}^\mathrm{unif}_p\in(0,\infty)$ such that
$$
\lim_{N\to\infty} N^{1/d} \,\underline{V}_{N,p}^\mathrm{rand}=\underline{\kappa}_p^\mathrm{unif}.
$$

\item[(ii)] Let $\mu$ be a probability measure on $\R^d$ that  has a finite $q$th moment for some $q>\frac{dp}{d-p}$ and suppose that 
$\frac{\dd \mu_a}{\dd \lambda^d}$ is Riemann integrable or $p=1$.  Then
\begin{align}\label{eq0201-1}
\limsup_{N\to\infty} N^{1/d} \, V_{N,p}^\mathrm{rand}(\mu) \le \kappa_p^\mathrm{unif}  \, \left(\int_{\R^{d}} 
\left(\frac{\dd \mu_a}{\dd \lambda^d}\right)^{1-\frac{p}{d}}\dd\lambda^d\right)^{1/p}
\end{align}
and
\begin{align}\label{eq0201-2}
\liminf_{N\to\infty} N^{1/d} \, V_{N,p}^\mathrm{rand}(\mu) \ge \underline{\kappa}_p^\mathrm{unif}  \, \left(\int_{\R^{d}} 
\left(\frac{\dd \mu_a}{\dd \lambda^d}\right)^{1-\frac{p}{d}}\dd\lambda^d\right)^{1/p}.
\end{align}
\end{itemize}
\end{Theorem}

\begin{Remark}
We conjecture that 
$\underline{\kappa}_p^\mathrm{unif} = \kappa_p^\mathrm{unif}$ in which case the inequality and $\limsup$ in \eqref{eq0201-1}  are actually an equality and $\lim$. Proving  the equality  $\underline{\kappa}_p^\mathrm{unif} = \kappa_p^\mathrm{unif}$ seems to be a general open problem in transport problems. Similar problems arise in 
Huesmann and Sturm in \cite{HS10} for optimal transports from Poisson point processes with Lebesgue intensity to Lebesgue measure. 
\end{Remark}

Let us compare our results with the classical high resolution formulas, see \cite[Theorem 6.2]{GL00}. The asymptotics of $V_{N,p}^\mathrm{opt}$ defined in 
\eqref{opt} satisfies
\begin{equation}\label{GL}
\lim_{N\to\infty} N^{1/d} \, V_{N,p}^\mathrm{opt}(\mu)= c_{p,d} \left(\int_{\R^{d}} 
\left(\frac{\dd \mu_a}{\dd \lambda^d}\right)^{\frac{d}{d+p}}\dd\lambda^d\right)^{1/d+1/p},
\end{equation}
whenever $\mu$ has a finite moment of order $q$ for some $q>p$. 
Here, the  constant $c_{p,d}$ is the corresponding limit for the uniform distribution 
on the unit cube in $\R^d$.  Its numerical value is known in a few special cases.

Theorem \ref{thm_pierce} can be used to improve \cite[Theorem 9.1(a)]{GL00}: there the validity of an asymptotic formula for the random quantization error is 
shown to be equivalent to the uniform integrability of $(N^{p/d}\min_{1 \le i \le N} \|X-Y_i\|^p)_{N \ge 1}$ where (for example) $X,Y_1,...$ are independent with
law $\mu$. Theorem \ref{thm_pierce} shows that uniform integrability holds provided that $1 \le p <d/2$ and $\mu$ has a finite moment of order $q$ for some 
$q>\frac {dp}{d-p}$.

Note that the integral term on the right hand side of (\ref{GL}) differs from the one in (\ref{eq0201-1}) and (\ref{eq0201-2}). This effect can be explained as follows: 
for a sequence of optimal codebooks $(\mathcal C(N))_{N\geq 1}$ of size $N$ the empirical measures $\frac 1N \sum_{x\in\mathcal C(N)}\delta _x$ tend to a 
measure that differs from $\mu$. In fact optimal codebooks allocate more points in the tails of the distribution. Since our approach does not account for 
such a correction, it is natural to expect a loss of efficiency for heavy tailed distributions. For arbitrary codebooks whose empirical distributions tend 
to the measure $\mu$, one has lower bounds which incorporate the same integral term as in our high resolution formula, see \cite[Thm.\ 7.2]{Der09}. 

A high resolution formula is also available for quantization with random codebooks and \emph{optimally} chosen weights. As a consequence of  \cite[Theorem 9.1(a)]{GL00} and Theorem~\ref{thm_pierce}, one has under the assumption of  Theorem~\ref{thm2} (without the Riemann integrability) equality in (\ref{eq0201-1}) for a different constant. Indeed, Theorem~\ref{thm_pierce} allows to verify an integrability assumption in \cite[Theorem 9.1(a)]{GL00} and thus to improve the result.
As a consequence, postprocessing of the weights can in the limit improve the error by a constant factor, irrespective the distribution~$\mu$.



\subsection{Preliminaries}


For a finite signed measure $\mu$ on the Borel sets of $\R^d$, we write $\|\mu\|:=|\mu|(\R^d)$ for its total variation norm (using the same symbol as 
for the norm on $\R^d$ should not cause any confusion). For finite (nonnegative) measures $\mu$ and $\nu$
we denote by $\mu\wedge \nu$ the largest measure 
that is dominated by $\mu$ and $\nu$. Furthermore, we set $(\mu-\nu)_+:=\mu-\mu\wedge \nu$.

Next, we introduce concatenation  of transports.
A transport $\xi$, i.e.\ a finite measure $\xi$ on $\R^d\times \R^d$, will be associated to a probability kernel $K$ and a measure $\nu$ on 
$\R^d$ via
\begin{align}\label{eq1911-1}
\xi(\dd x,\dd y)=\nu(\dd x) K(x,\dd y),
\end{align}
so $\nu$ is the first marginal of $\xi$.
We call $\xi$ the transport with source $\nu$ and kernel $K$. Let $\cK$  denote the set of probability kernels from 
$(\R^d,\mathcal{B}^d)$ into itself and
consider the semigroup $(\cK,*)$, where the operation $*$ is defined via 
$$
K_1*K_2 (x,A):=\int K_1(x,\dd z) K_2(z,A) \qquad(x\in\R^d, A\in\mathcal{B}^d)
$$
Now we can iterate transport schedules: Let $\nu_0,\dots,\nu_n$ be measures on $\R^d$ with identical total mass and 
let $\xi_k\in\cM(\nu_{k-1},\nu_k)$. Then the concatenation of the transports $\xi_1,\dots,\xi_n$  is formally the transport 
described by the source $\nu_0$ and the probability kernel $K=K_1*\dots*K_n$, where $K_1,\dots,K_n$ are the kernels 
associated to $\xi_1,\dots,\xi_n$. Note that the relation (\ref{eq1911-1}) defines the kernel uniquely up to $\nu$-nullsets 
so that the concatenation of transport schedules is a well-defined operation on the set of transports. In analogy to the 
operation $*$ on  $\cK$, we write $\xi_1*\dots *\xi_n$ for the concatenation of the transport schedules.

We summarize elementary properties of the Wasserstein metric in a lemma.

\begin{Lemma}\label{wasser_prop}
Let $\xi, \mu,\mu_1,\dots$ and $\nu,\nu_1,\dots$ be finite measures on $\R^d$ such that $\|\xi\| = \|\mu\| =\| \nu\|$.
\begin{enumerate}
\renewcommand{\labelenumi}{\textbf{(\roman{enumi})}}
	\item \textbf{Convexity:} Suppose that $\mu=\sum_{k\in\N}\mu_k$ and $\nu=\sum_{k\in\N} \nu_k$ and that for all 
$k\in\N$, $\|\mu_k\|=\|\nu_k\|$. Then
	\begin{equation}
	\rho_{p}^p\left(\mu,\nu\right) \leq \sum_{k=1}^{\infty}\rho_{p}^{p}\left(\mu_k,\nu_k\right).
	\end{equation}
	
	\item \textbf{Triangle-inequality:} One has
	\begin{equation}
	\rho_{p}\left(\mu,\nu\right) \leq \rho_{p}(\mu,\xi) +\rho_p(\xi,\nu).
	\end{equation}
	\item \textbf{Translation and scaling}: Let $T:\R^d \to \R^d$ be a map, which consists of a translation  and a scaling by 
the factor $a>0$. Then
	\begin{equation}
	\rho_{p}\left( \mu \circ T^{-1}, \nu \circ T^{-1} \right)=a\,\rho_{p}\left( \mu, \nu \right).
	\end{equation}
\end{enumerate}
\end{Lemma}

\section{Proof of the Pierce type result}\label{sec_pierce}

In order to prove Theorem~\ref{thm_pierce}, we first derive an estimate for  general distributions on the unit cube $[0,1)^d$.

\begin{Proposition}\label{coarse bound new}
 Let  $1\le p<\frac{d}{2}$. There exists a constant $\kappa_p^\mathrm{cube}\in (0,\infty)$ such that for any probability measure 
$\mu$ on $[0,1)^d$ and $N\in\N$
$$
 V_{N,p}^\mathrm{rand}(\mu) \leq \kappa^{\mathrm{cube}}_p\, N^{-\frac{1}{d}}.
$$
\end{Proposition}

\begin{Remark}
The constant $\kappa^\mathrm{cube}_p$ is explicit. Let $\mathfrak{d}=\sup_{x,y\in[0,1)^d}\|x-y\|$ denote the diameter of $[0,1)^d$. Then
\[
 \kappa^{\mathrm{cube}}_p=\mathfrak{d} \,2^{\frac{d-2}{2p}}  \Bigl[ \frac{1}{1-2^{p-\frac{d}{2}}} + \frac{1}{1-2^{-p}} \Bigr]^{\frac{1}{p}}.
\]
\end{Remark}

For the proof of Proposition \ref{coarse bound new} we use a nested sequence of partitions of $B=[0,1)^d$. 
Note that $B$ can be partitioned into $2^d$ translates  
$B_1,\dots,B_{2^d}$  of $2^{-1} B$. We iterate this procedure and partition each set $B_k$ into $2^d$ translates 
$B_{k,1},\dots,B_{k,2^d}$ of $2^{-2} B$. We  continue this scheme obeying the rule that each set $B_{k_1,\dots,k_l}$ is partitioned 
into $2^d$ translates $B_{k_1,\dots,k_l,1},\dots, B_{k_1,\dots,k_l,2^d}$ of $2^{-(l+1)}B$.  
These translates of $2^{-l} B$ form a partition of $B$ and we denote this 
collection of sets by $\cP_l$, the $l$th level. We now endow the sets $\cP:=\bigcup_{l=0}^\infty \cP_l$ with a $2^d$ary tree 
structure. $B$ denotes the root of the tree and the father of a set $C\in \cP_l$ ($l\in\N$) is the unique set $F \in \cP_{l-1}$ 
that contains  $C$.

\begin{Lemma} \label{le0301-1} Let $\mu$ and $\nu$ be two probability measures supported on $B$ such that for all $C\in\cP$
$$
\nu(C)>0 \ \Rightarrow \ \mu(C)>0.
$$
Then
$$
\rho_p^p(\mu,\nu) \le \frac12  \mathfrak{d}^p \sum_{l=0}^\infty 2^{-pl} \sum_{F\in\cP_l} \,\sum_{C \mathrm{child\, of} F} \Bigl|\nu(C)-\nu(F) 
\frac{\mu(C)}{\mu(F)}\Bigr|
$$
with the convention that $\frac 00=0$. 
\end{Lemma}

For the proof we use couplings defined via partitions. Let $(A_k)$ be a  (finite or countably infinite) 
Borel partition of the Borel set $A\subset\R^d$. 
For two finite measures $\mu_1,\mu_2$ on $A$ with equal masses, we call the measure $\nu$ defined by 
$$
\nu\big|_{A_k} =\frac{\mu_2(A_k)}{\mu_1(A_k)} \,\mu_1\big|_{A_k}
$$
the \emph{$(A_k)$-approximation} of $\mu_1$ to $\mu_2$ provided that it is well defined (i.e.~that $\mu_1(A_k)=0$ implies 
$\mu_2(A_k)=0$).

The $(A_k)$-approximation $\nu$ is associated to a transport from $\mu_1$ to $\nu$.
Note that $$(\mu_1\wedge \nu)\big|_{A_k}= \frac{\mu_1(A_k)\wedge \mu_2(A_k)}{\mu_1(A_k)} \,\mu_1\big|_{A_k}$$ and we define a transport 
$\xi\in \cM(\mu_1,\nu)$ via 
$$
\xi=(\mu_1\wedge \nu)\circ \psi^{-1} + \frac 1\delta (\mu_1-\nu)_+\otimes (\nu-\mu_1)_+ 
$$
where $\delta:=\frac12 \sum_k |\mu_1(A_k)- \mu_2(A_k)|$ and  $\psi:\R^d\to \R^d \times \R^d, x\mapsto(x,x)$. Then  
$$\xi(\{(x,y)\in\R^d\times \R^d: x\not=y\}) = \delta.
$$

\begin{proof}[ of Lemma~\ref{le0301-1}]
For $l\in\N_0$, we set
$$
\mu_l=\sum_{A\in \cP_l} \frac{\nu(A)}{\mu(A)}\, \mu|_A
$$
which is the $\cP_l$-approximation of $\mu$ to $\nu$.
By construction, one has for each set $F\in \cP_l$ with $l\in\N_0$
$$
\mu_l(F)=\mu_{l+1}(F).
$$
Moreover, provided that $\mu_l(F)>0$, one has for each child $C$ of $F$
$$
\mu_{l+1}|_C= \frac {\nu(C)} {\mu(C)} \,\mu|_C= \frac {\mu(F) \nu(C)} {\mu(C) \nu(F)}\, \mu_{l}|_C 
$$
so that $\mu_{l+1}|_F$ is the $\{C\in \cP_{l+1}: C\subset A\}$-approximation of $\mu_l|_F$ to $\nu|_F$. Hence, there exists a transport $\xi_F\in\cM(\mu_l|_F,\mu_{l+1}|_F)$ with
\begin{align}\label{eq0301-1}
\xi_F (\{(x,y): x\not=y\}) = \frac 12\sum_{C\text{ child of }F} \Big\lvert\nu(C)-\nu(F) \frac{\mu(C)}{\mu(F)}\Big\rvert.
\end{align}
Since each family $\cP_l$ is a partition of the root $B$, we have
$$\xi_{l+1}:=\sum_{F\in \cP_l} \xi_F\in\cM(\mu_l,\mu_{l+1}).$$  
Next, note that
$\rho_p(\mu_l,\nu)\le \mathfrak {d} 2^{-l}$
so that $\mu_l$ converges in the $p$th Wasserstein metric to $\nu$ which implies that
\begin{equation}\label{Grenz}
\rho_p(\mu,\nu)\le \sup_{l\in\N} \rho_p(\mu,\mu_l).
\end{equation}
The concatenation of the  transports $(\xi_l)_{l\in\N}$ leads to new transports
$$
\xi^{l}=\xi_1*\dots*\xi_l\in\cM(\mu,\mu_l)
$$
Each of the transports $\xi_k$ is associated to a  kernel $K_k$ and,  by Ionescu-Tulcea, there exists a sequence $(Z_l)_{l\in\N_0}$ of 
$[0,1)^d$-valued random variables with 
$$
\IP(Z_0\in A_0,\dots , Z_l\in A_l)= \int_{A_0} \int_{A_1} \dots \int_{A_{l-1}} K_l(x_{l-1},A_l)\, \dots \, K_1(x_0,\dd x_1) \, \mu(\dd x_0)
$$ for every $l\in\N$. 
Then the joint distribution of $(Z_0,Z_l)$ is $\xi^{l}$.
Let 
$$
L=\inf \{l\in N_0: Z_{l+1}\not =Z_{l}\}
$$
and note that all entries $(Z_l)_{l\in\N_0}$ lie in one (random) set $A\in \cP_L$, if  $\{L<\infty\}$ enters, and are identical on $\{L=\infty\}$. 
Hence, for any $k\in\N$
\begin{align*}
 \E[\|Z_0-Z_k\|^p ]& \le \mathfrak{d}^p \,\E[ 2^{-pL}] \le \mathfrak{d}^p \, \sum_{l=0}^\infty  2^{-pl}\, \IP(Z_{l+1}\not =Z_l) \\
&= \mathfrak{d}^p \, \sum_{l=0}^\infty  2^{-pl} \xi_{l+1} (\{(x,y):x\not=y\}) \\
&= \frac12\mathfrak{d}^p \, \sum_{l=0}^\infty  2^{-pl} \sum_{F\in\cP_l} \sum_{C\text{ child of }F}  \Big\lvert\nu(C)-\nu(F) \frac{\mu(C)}{\mu(F)}\Big\rvert, 
\end{align*}
where we used (\ref{eq0301-1}) in the last step, so the assertion follows by \eqref{Grenz}.
\end{proof}

\begin{proof}[ of Proposition \ref{coarse bound new}] We apply the above lemma with 
$\nu = \hat\mu_N$.
Hence,
\begin{align}\label{eq0301-3}
\rho_p^p(\mu,\hat\mu_N) \le  \frac12\mathfrak{d}^p \, \sum_{l=0}^\infty  2^{-pl} \sum_{F\in\cP_l} \sum_{C\text{ child of }F}  
\Big\lvert\hat\mu_N (C)-\hat\mu_N(F) \frac{\mu(C)}{\mu(F)}\Big\rvert.
\end{align}
Note that conditional on the event $\{N \,\hat \mu_N(F)=k\}$ $(k\in \N)$ the random vector $(N\hat \mu_N(C))_{C\text{ child of }F}$ 
is multinomially distributed with parameters $k$ and success probabilities $(\mu(C)/\mu(F))_{C\text{ child of }F} $. Hence,
\begin{align*}
\E\Bigl[\sum_{C\text{ child of } F} &\Bigl|\hat \mu_N(C)-\hat\mu_N(F) \frac{\mu(C)}{\mu(F)}\Bigr| \Big| N\hat\mu_N(F)=k  \Bigr] \\
&=\frac 1N \,\E\Bigl[\sum_{C\text{ child of } F} \Bigl| N \,\hat \mu_N(C)-k\, \frac{\mu(C)}{\mu(F)}\Bigr| \Big| N\hat\mu_N(F)=k  \Bigr]\\
&\le \frac 1N\,\sum_{C\text{ child of } F} \mathrm{var}\bigl(  N \,\hat \mu_N(C)
\bigr| \big| N\hat\mu_N(F)=k  \bigr)^{1/2}\\
&\le  \frac{\sqrt{k}}{N}\sum_{C\text{ child of } F} \sqrt{\frac{\mu(C)}{\mu(F)}} 
\le {2^{\frac{d}{2}}}   \frac{\sqrt k}{N},
\end{align*}
where we used Jensen's inequality in the last step.
We set $\zeta(t)=\sqrt t\wedge t$ ($t\ge0$) and observe that
$$
 \E\Bigl[\sum_{C\text{ child of } F} \Bigl|\hat \mu_N(C)-\hat\mu_N(F) \frac{\mu(C)}{\mu(F)}\Bigr|\Bigr]   \le \frac{2^{\frac{d}{2}}}{N}\,\zeta(N \mu(F)).
$$
Consequently, it follows from  (\ref{eq0301-3}) and Jensen's inequality that
$$
\E[\rho_p^p(\mu,\hat\mu_N)] \le \frac12 \mathfrak{d}^p\sum_{l=0}^\infty2^{-pl} \sum_{F\in\cP_l} \frac{2^{\frac{d}{2}}}{N}\zeta(N \mu(F))
\le\mathfrak{d}^p 2^{\frac{d}{2}-1} N^{-1} \sum_{l=0}^\infty  2^{(d-p)l} \zeta(2^{-dl}N). 
$$
Let $l^{*}:=\lfloor \log_2 N^{\frac{1}{d}}\rfloor$. Then,
\begin{align*}
\E[\rho_p^p(\mu,\hat\mu_N)] &\le \mathfrak{d}^p 2^{\frac{d}{2}-1} N^{-1} \Bigl[  \sum_{l=0}^{l^{*} } 2^{(\frac12 d-p)l} \sqrt N 
+\sum_{l= l^{*}+1}^\infty  2^{-pl} N\Bigr] \\
&\le \mathfrak{d}^p 2^{\frac{d}{2}-1} N^{-1} \Bigl[ \sum_{k=0}^{\infty} 2^{(\frac{d}{2}-p)(l^{*}-k)} \sqrt{N} 
+2^{-p(l^{*}+1)} \sum_{j=0}^\infty  2^{-pj} N \Bigr]  \\
&\le \mathfrak{d}^p 2^{\frac{d}{2}-1} N^{-\frac{p}{d}} \Bigl[ \frac{1}{1-2^{p-\frac{d}{2}}} + \frac{1}{1-2^{-p}} \Bigr],
\end{align*}
so the assertion follows.
\end{proof}

We are now in the position to prove Theorem \ref{thm_pierce}. 
Since all norms on $\R^d$ are equivalent, it suffices to prove the result for the maximum norm $\|.\|_{\mathrm{max}}$.

\begin{Theorem}\label{Pierce}
Let $p\in[1,\frac d2)$ and $q > \frac{pd}{d-{p}}$. One has for any probability measure $\mu$ on $\R^d$ that 
\begin{align}\label{eq2111-1var}
V_{N,p}^{\mathrm{rand}}(\mu) \le \kappa_{p,q}^{\mathrm{Pierce}} \, \left[ \int_{\R^d} \left\|x\right\|^q_{\mathrm{max}}\,\dd\mu(x) \right]^{1/q} \, N^{-1/d},
\end{align}
where 
$\kappa_{p,q}^{\mathrm{Pierce}} = \kappa_p^{\mathrm{cube}} \, 
\left[  \frac{2^{p-1} 2^{\frac{q}{2}} \cd^p}{1-2^{p-\frac{1}{2}q}} + 
\frac{2^{p+q(1-p/d)}(\kappa_p^{\mathrm{cube}})^p}{1- 2^{-q(1-p/d)+p}}\right]^{1/p}$.
\end{Theorem}

\begin{proof}
By the scaling invariance of inequality~(\ref{eq2111-1var}), we can and will assume without loss of generality that 
$\int \|x\|_{\mathrm{max}}^q\,\dd\mu(x) =1$. We partition $\R^d$ into a sequence of sets $(B_n)_{n\in\N_0}$ defined as
$$
B_0:=B:=[-1,1)^d \ \text{ and } \ B_n:=(2^n B) \backslash (2^{n-1} B) \text{ for }n\in\N.
$$
We denote by  $\nu$ the random $(B_n)$-approximation of $\mu$ to $\hat\mu_N$, that is
$$
\nu|_{B_n} =\frac{\hat\mu_N(B_n)}{\mu(B_n)} \mu|_{B_n} \ \text{ for } \ n\in\N_0. 
$$
Then $\xi=(\mu\wedge \nu)\circ \psi^{-1}+  \delta^{-1}\, (\mu-\nu)^+\otimes (\nu-\mu)^+$ with $\delta:=|(\mu-\nu)^+|=|(\nu-\mu)^+|$ and 
$\psi:\R^d \to \R^d\times\R^d, x\mapsto (x,x)$ defines a transport in $\cM(\mu,\nu)$, such that
\begin{align*}
\int \|x-y\|^p \, \xi(\dd x,\dd y) &= \delta^{-1}\int_{\R^d} \int_{\R^d} \|x-y\|^p \,  (\mu-\nu)^+ (\dd x)\,  (\nu-\mu)^+(\dd y) \\
&\le 2^{p-1} \int_{\R^d} \|x\|^p\, (\mu-\nu)^+ (\dd x)+ 2^{p-1}  \int_{\R^d} \|y\|^p\, (\nu-\mu)^+ (\dd y) \\
&\le 2^{p-1} \sum_{n=0}^{\infty}\int_{B_n} \|x\|^p\, (\mu-\nu)^+ (\dd x)+ 2^{p-1}  \sum_{n=0}^{\infty}\int_{B_n} \|y\|^p\, (\nu-\mu)^+ (\dd y ) \\
&\le 2^{p-1} \sum_{n=0}^{\infty} \cd^p 2^{np} \cdot \left|\mu-\nu\right|(B_n).
\end{align*}
Note that $N\hat{\mu}_N(B_n)\sim  \mathrm{Bin}(N,\mu(B_n))$ and that by the Markov inequality
\begin{align}\label{ieq2111-2}
\mu(B_n)\leq 2^{-q(n-1)} \int \|x\|_\mathrm{max}^q \,\dd \mu(x)= 2^{-q(n-1)}. 
\end{align}
The inequality remains true for $n=0$. Thus
\begin{equation}\label{ieq1612-1}
\begin{split}
&\hspace{-2pt} \E [\rho^p_p(\mu,\nu)] 
\leq \sum_{n=0}^{\infty} 2^{p-1} 2^{np}\cd^p \,\E \left[ \left|\mu(B_n) - \hat{\mu}_N(B_n)\right| \right] \\
&\leq \sum_{n=0}^{\infty} 2^{p-1}2^{np}\cd^p \, N^{-\frac{1}{2}}  \mu(B_n)^{\frac{1}{2}} \\
&\leq 2^{p+\frac q2-1}  \cd^p  N^{-\frac{1}{2}} \sum_{n=0}^{\infty} 2^{n(p-\frac 12q)}  = \frac{2^{p+\frac{q}{2}-1}}{1-2^{p-\frac{1}{2}q}}  \cd^p  
N^{-\frac{1}{2}}.
\end{split}
\end{equation}

It remains to analyse $\E [\rho^p_p(\nu,\hat\mu_N)]$. Given that $\{N\hat{\mu}_{N}(B_n)=k\}$ the random measure 
$\frac Nk\hat\mu_{N}|_{B_n}$ is the empirical measure of $k$ independent $\frac{\mu|_{B_n}}{\mu(B_n)}$-distributed 
random variables. By Lemma \ref{wasser_prop} (i) and Proposition~\ref{coarse bound new},
\begin{align*}
\E[ \rho_p^p(\nu,\hat\mu_N) ]&\le  \sum_{n=0}^{\infty} \E[\rho_p^p(\nu|_{B_n},\hat\mu_N|_{B_n}) ]\\
&\le  \sum_{n=0}^{\infty} \sum_{k=1}^\infty \IP(N \hat{\mu}_{N}(B_n)=k)\, 2^{(n+1)p}\, \frac kN (\kappa_p^{\mathrm{cube}})^p k^{-p/d}.
\end{align*}
Using that $\E \left[\hat{\mu}_{N}(B_n)\right]= \mu(B_n)$, we conclude with  Jensen's inequality that
$$
\E[ \rho_p^p(\nu,\hat\mu_N) ]\le  (\kappa_p^{\mathrm{cube}})^p N^{-p/d} \sum_{n=0}^{\infty} 2^{(n+1)p}\, \mu(B_n)^{1-p/d}.
$$
We use again inequality (\ref{ieq2111-2}) to derive 
\begin{align*}
\E[ \rho_p^p(\nu,\hat\mu_N) ] &\le  (\kappa_p^{\mathrm{cube}})^p N^{-p/d} \sum_{n=0}^\infty 2^{(n+1)p-q(n-1)(1-\frac pd)} \\
&= (\kappa_p^{\mathrm{cube}})^p  \frac{2^{p+q(1-p/d)}}{1- 2^{-q(1-p/d)+p}} \,  N^{-p/d} .
\end{align*}
Note that $\frac pd\le \frac12$ and altogether, we finish the proof by applying the triangle inequality (property (ii) of 
Lemma~\ref{wasser_prop}) and equation (\ref{ieq1612-1}) to deduce that
$$
\E[\rho^p_p(\mu,\hat\mu_N)]^{1/p} 
\le \underbrace{\left[  \frac{2^{p-1} 2^{\frac{q}{2}} \cd^p}{1-2^{p-\frac{1}{2}q}} + 
\frac{2^{p+q(1-p/d)}(\kappa_p^{\mathrm{cube}})^p}{1- 2^{-q(1-p/d)+p}}\right]^{1/p}}_{=:\kappa_{p,q}^{\mathrm{Pierce}}} \, N^{-1/d}.
$$
\end{proof}

\section{Asymptotic analysis of the uniform measure}\label{sec_uni}

Next, we investigate the asymptotics of the random quantization of the uniform distribution $\cU$ on the unit cube $B=[0,1)^d$. 
The aim of this subsection is to prove the existence of the limits 
$$ \kappa_p^{\mathrm{unif}}:=\lim_{N\to\infty} N^{1/d} \, V^{\mathrm{rand}}_{N,p}(\cU),\qquad 
\underline{\kappa}_p^{\mathrm{unif}}:=\lim_{N\to\infty} N^{1/d} \, \underline{V}^{\mathrm{rand}}_{N,p}$$
which is the first statement of Theorem~\ref{thm2}.

\begin{Notation}
Let $A$ and  $S$ denote two sets with  $A\subset S$ and suppose that  $v=(v_j)_{j=1,\dots,N}$ is an $S$-valued vector. We call the 
vector $v_A$ consisting of all entries of $v$ in $A$ the $A$-\emph{subvector} of $v$, that is
$$v_A:=(v_{\gamma(j)})$$ where $(\gamma(j))$ is an enumeration of the entries of $v$ in $A$.
\end{Notation}

For a Borel set $A$ with finite nonvanishing Lebesgue measure, we denote by $\cU(A)$ the uniform distribution on $A$.  The proof of the 
existence of the limit makes use of the following lemma.

\begin{Lemma}\label{divide}
Let $K\in\N$ and let $A,A_1,\dots , A_K\subset\R^d$ be Borel sets such that $\lambda^d(A)<\infty$ and that the sets  $A_1,\dots,A_K\subset\R^d$ 
are pairwise disjoint and cover $A$. 
Fix $N\in\N$ and suppose that $\xi_k:=N\cdot\frac{\lambda^d(A_k\cap A)}{\lambda^d(A)}\in\N_0$ for  $k=1,\dots,K$. 

Assume that $X=\left(X_1,\dots,X_N\right)$ is a random vector consisting of independent  $\mathcal{U}(A)$-distributed entries. Then one can 
couple $X$ with  a random vector $Y=(Y_1,\dots,Y_N)$  which has  $A_k$-subvectors consisting of $\xi_k$ independent $\cU(A_k)$-distributed 
entries such that the individual subvectors are  independent  and such that
\begin{equation}
\E\Bigl[ \sum_{j=1}^{N} \ind_{\{X_j\neq Y_j\}}\Bigr] \leq \frac{\sqrt{K}\sqrt N}{2}.
\end{equation}
\end{Lemma}

\begin{proof}
For $k=1,\dots,K$, denote by $X^{(k)}$ the $A_k$-subvector of $X$. 
For each $k$ with $\xi_k\le \mathrm{length}(X^{(k)})$, we keep the first $\xi_k$ entries of $X$ in $A_k$ and erase the remaining ones. 
For any other $k$'s, we fill up $\xi_k- \mathrm{length}(X^{(k)})$ of the empty places by independent $\cU(A_k)$-distributed elements. 
Denoting the new vector by $Y$, we see that $Y$ has $A_k$-subvectors of length $\xi_k$. 
Clearly, $Y$ has independent subvectors that are uniformly distributed on the respective sets. 
Since the length of the $A_k$-subvector is binomially distributed with parameters $N$ and 
$q_k:=\frac{\lambda^d(A_i\cap A)}{\lambda^d(A)}$, we get 
\begin{align*}
 \E\Bigl[ \sum_{j=1}^{N} \ind_{\{X_j\neq Y_j\}}\Bigr]  &= \frac{1}{2}\E\Bigl[ \sum_{k=1}^{K} \bigl|\mathrm{length}(X^{(k)}) -\xi_k\bigr|\Bigr]  
\leq \frac{1}{2}\sum_{i=1}^{K} \textrm{var}\bigl(\mathrm{length}({X^{(k)}})\bigr)^{1/2} \\
&   \leq \frac{1}{2}\sqrt{N}\sum_{k=1}^{K} \sqrt{q_k}  \leq \frac{1}{2}\sqrt{K}\sqrt{N}.
\end{align*}
\end{proof}

\begin{proof}[ of the first statement of (i) of Theorem \ref{thm2}]
Let $M\in\N$ be arbitrary but fixed. Further, let $N\in\N$, $N>2^dM$, and denote by $B_0=\left[0,a\right)^d$, $a^d=\frac{M}{N}$, 
the cube with volume $\lambda^d\left(B_0\right)=\frac{M}{N}$. We divide $\left[0,1\right)^d$ into two parts, the main one 
$B^{\mathrm{main}}:=\left[0,\left\lfloor 1/a\right\rfloor a\right)^d$ and the remainder 
$B^\mathrm{rem}:=\left[0,1\right)^d\backslash B^\mathrm{main}$. Note that $\lambda^d(B^\mathrm{rem}) \to 0$ as $N\to\infty$. 
We represent $B^\mathrm{main}$ as the  union of $n=\lfloor a^{-1} \rfloor^d$ pairwise disjoint translates of 
$B_0$: $$B^\mathrm{main}=\cup_{k=1}^{n} B_k.$$

Let $X=(X_1,\dots,X_N)$  denote a vector of $N$ independent $\mathcal{U}[0,1)^d$-distributed entries. We shall now couple $X$ with 
 a random vector $Y=(Y_1,\dots,Y_N)$ in such a way that most of the entries of $X$ and $Y$ coincide and such that  the $B_k$-subvectors 
are independent and consist of $M$ independent $\cU(B_k)$-distributed entries.
To achieve this goal we successively  apply Lemma \ref{divide} to construct  random vectors $X^0,\dots ,X^L$ and finally set $X^{L}=Y$. 
First we apply the coupling for $X$ with the decomposition $[0,1)^d= B^\mathrm{main}\, \dot \cup\,  B^\mathrm{rem}$
and denote by $X^0$ the resulting vector. In the next step a  $2^d$ary tree $\cT$  with leaves being the boxes $B_1,\dots B_n$ is used 
to define further couplings. We let $L$ denote the smallest integer with $2^L B_0\supset B^\mathrm{main}$, i.e.\ $L=\lceil -\log_2 a\rceil$,
and set
$$\cT_l:=\{  \gamma  + 2^{L-l} B_0  : \gamma\in (2^{L-l}a \Z^d )\cap B^\mathrm{main}\}$$
for $l=0,\dots,L$. Now $\cT$ is defined as the rooted tree which has at level $l$ the boxes (vertices) $\cT_l$ and a box 
$A_\mathrm{child}\in\cT_l$ is the child of a box $A_\mathrm{parent}\in\cT_{l-1}$ if $A_\mathrm{child}\subset A_\mathrm{parent}$.

We associate the vector $X^0$ with the $0$th level of the tree. Now we define consecutively $X^1,\dots,X^L$ via the following rule. 
Suppose that $X^l$ has already been defined. For each $A\in \cT_l$ we apply the above coupling independently to the $A$-subvector of 
$X^l$  with the representation
$$
A= \dot\bigcup_{B\text{ child of  }A} B.
$$
By induction, for each $A\in\cT_l$, the $A$-subvector of $X^l$ consists of  $N\lambda^d(A)\in\N$ independent $\mathcal{U}(A)$-distributed 
random variables. In particular, this is valid for the last level $Y=X^L$.

We proceed with an error analysis.  Fix $\om\in\Om$ and $j\in\{1,\dots,N\}$ and suppose that $X^0_j(\om),\dots, X^L_j(\om)$ is altered in 
the step $l\to l+1$ for the first time and that $X^0_j(\om) \in B\in \cT_l$. Then it follows that $X^L_j(\om)\in B$ so that
$$
\left\|X^0_j(\om)- X^L_j(\om)\right\|\le \mathrm{diameter}(B) \le  a \cd  2^{L-l},
$$ 
where $\cd$ is the diameter of $[0,1)^d$.
Consequently,
\begin{align*}
\E \Bigr[\sum_{j=1}^N \left\|X^0_j -X^L_j\right\|^p\Bigr] \le \E \Bigl[\sum_{j=1}^N  \sum_{l=0}^{L-1} \ind_{\{X^l_j\not = X^{l+1}_j\}} (a \cd  2^{L-l})^p  \Bigr].
\end{align*}
By Lemma \ref{divide} and the Cauchy-Schwarz inequality, one has, for $l=1,\dots,L$,
\begin{align*}
\E \Bigl[\sum_{j=1}^N  \ind_{\{X^l_j\not = X^{l-1}_j\}} \Bigr] \le  \frac{1}{2} \sqrt{2^d} \sqrt N \sum_{A\in\cT_{l-1}} \sqrt{ \lambda^d(A)}  
\le \frac12 \,2^{dl/2} \sqrt N.
\end{align*}
Together with the former estimate we get
\begin{align*}
 \E \Bigl[\sum_{j=1}^N \lVert X^0_j -X^L_j\rVert^p\Bigr] & \le \frac12 \,(a\cd)^p \sqrt N \sum_{l=1}^{L} 2^{(L-l)p+dl/2}
 \leq \frac12 \frac{(a\cd)^p}{1-2^{-\frac d2+p}}\,2^{dL/2} \sqrt{N} 
\end{align*}
Next, we use that $a=(\frac MN)^{1/d}$ and $2^L\le \frac2a$ to conclude that
$$
 \E \Bigl[\sum_{j=1}^N \lVert X^0_j -X^L_j\rVert^p\Bigr] \le \frac{2^{d/2-1} \cd^p}{1-2^{-d/2+p}} \,M^{\frac{p}{d} - \frac{1}{2}} \, N^{1 - \frac{p}{d}}.
$$
Hence, there exists a constant $C$ that does not depend on $N$ and $M$ such that
\begin{align}\begin{split}\label{eq2812-1}
\E \Bigl[\frac 1N \sum_{j=1}^N \lVert X_j -Y_j\rVert^p\Bigr]^{1/p} &\le \E \Bigl[\frac 1N \sum_{j=1}^N \lVert X_j-X^0_j\rVert^p\Bigr]^{1/p} + \E \Bigl[\frac 1N \sum_{j=1}^N \lVert X^0 _j -X^L_j\rVert^p\Bigr]^{1/p}\\
&\le C \,  \bigl[ N^{-\frac1{2p}}+M^{-(\frac1{2p}-\frac 1d)} \, N^{-\frac 1d}  \bigr].
\end{split}\end{align}

By construction, $Y$ has for each $k=1,\dots,n$, a $B_k$-subvector of $M$ independent $\cU(B_k)$-distributed random 
variables and we denote the corresponding empirical measure by $\hat \mu^{(k)}_M$. Morever, its $B^\mathrm{rem}$-subvector contains 
$N-nM$ independent $\cU(B^\mathrm{rem})$-distributed entries and we denote its empirical measure by $\hat \mu^\mathrm{rem}_{N-nM}$.
Letting $\hat \mu^Y_N$ denote the empirical measure of $Y$, we conclude with Lemma \ref{wasser_prop} and Proposition \ref{coarse bound new} that  
\begin{align}\label{eq2812-neu}
N\, \E[\rho_p^p(\hat \mu^Y_N, \cU)] &\le \sum_{k=1}^n M\,  \E[\rho_p^p(\hat \mu^{(k)}_M, \cU(B_k))] + (N-nM) \, 
\E[\rho_p^p(\hat \mu^{\mathrm{rem}}_{N-nM}, \cU(B^\mathrm{rem}))]\nonumber \\
&\le n M a^p (V^\mathrm{rand}_{M,p}(\cU))^p + (\kappa^\mathrm{cube}_p )^p (N-nM)^{1-p/d}.
\end{align}

Next, we let $N$ tend to infinity and combine the above estimates. Note that $N^{1/d} a \to M^{1/d}$ and 
$\frac{nM}{N}\to 1$ so that
$$
\limsup_{N\to \infty} N^{1/d} \,  \E[\rho^p_p(\hat \mu^Y_N, \cU)]^{1/p} \le M^{1/d} V^\mathrm{rand}_{M,p} (\cU).
$$ 
Moreover, (\ref{eq2812-1}) implies that
$$
\limsup_{N\to\infty} N^{1/d} \,  \E[\rho^p_p(\hat \mu^X_N, \hat \mu^Y_N)]^{1/p} \le C \,M^{-(\frac 1{2p}-\frac1d)}.
$$
Now fix $\eps\in(0,1]$ arbitrarily and let $M\ge \frac 1\eps$ such that
$$
M^{1/d}\, V^{\mathrm{rand}}_{M,p} (\cU)\le \liminf_{N\to\infty} N^{1/d}\,  V^{\mathrm{rand}}_{N,p} (\cU)+\eps.
$$
Then 
\begin{align*}
\limsup_{N\to\infty} N^{1/d}\,V^{\mathrm{rand}}_{N,p}(\cU) &\le  M^{1/d} V^\mathrm{rand}_{M,p}(\cU)+ C \,M^{-(\frac 1{2p}-\frac1d)}\\ 
&\le \liminf_{N\to\infty} N^{1/d}\,  V^{\mathrm{rand}}_{N,p}(\cU)+\eps +C \, \eps ^{\frac 1{2p}-\frac1d}
\end{align*}
and letting $\eps\downarrow 0$ finishes the proof.
\end{proof}

\begin{proof}[ of the second statement of (i) of Theorem \ref{thm2}]
The proof of the second statement is very similar to the proof of the first statement. The crucial difference  is that the arguments are now based on superadditivity compared to the subadditivity of the Wasserstein metric (in the sense of part (i) of Lemma 
\ref{wasser_prop}) that was used in the proof of the first statement.

 We now look at a nonsymmetric  modified version of the Wasserstein distance that allows leakage at the boundaries. For two probability measures $\nu_1$ and $\nu_2$ on $[0,1]^d$, we define
$$
\underline{\rho}_{p}(\nu_1,\nu_2):=\inf_{\nu_1'\in \Lambda(\nu_1)} \rho_p(\nu_1',\nu_2),
$$
where $\Lambda(\nu_1)$ denotes all probability measures $\zeta$ on  $[0,1]^d$ 
which satisfy $\zeta(A) \le \nu_1(A)$ for all Borel sets $A$ in $(0,1)^d$. 

We make use of thee same notation as in the proof of the first statement. First note that similar as in (\ref{eq2812-neu})
$$
N \,\E[\underline{\rho}_{p}^p(\mathcal U,\hat\mu^Y_N)]\geq nM a^p (\underline{V}^\mathrm{rand}_{M,p})^p
$$
Since, in general,
$$
\underline{\rho}_{p}(\mathcal U,\hat\mu^Y_N)
\leq \underline{\rho}_{p}(\mathcal U,\hat\mu^X_N)+ \rho_p(\hat \mu^X_N,\hat\mu_N^Y),
$$
we conclude that
\begin{align*}
\liminf_{N\to\infty}N^{1/d}  \E[\underline{\rho}^p_{p}(\mathcal U,\hat\mu^X_N)]^{1/p}&\geq \liminf_{N\to\infty}N^{1/d}  \E[\underline{\rho}^p_{p}(\mathcal U,\hat\mu^Y_N)]^{1/p}-\limsup_{N\to\infty}N^{1/d}  \E[{\rho}^p_{p}(\hat\mu^X_N,\hat\mu^Y_N)]^{1/p}  \\
&\geq \liminf_{N\to\infty} N^{1/d} (nM/N  )^{1/p} a \underline{V}_{M,p}^\mathrm{rand}(\cU) -C \,M^{-(\frac 1{2p}-\frac1d)}\\
&\geq M^{1/d} \,\underline{V}_{M,p}^\mathrm{rand}(\cU) -C \,M^{-(\frac 1{2p}-\frac1d)}.
\end{align*}
The proof is finished as above.
%
%
\end{proof}

\section{Proof of the high resolution formula}\label{sec_hrf}
\subsection{Proof  of the high resolution formula for general $p$}

\begin{Definition}
We call a finite measure $\mu$ on $\R^d$ \textit{approachable from below}, if there exists for any $\eps>0$ a finite number of 
cubes $B_1,\dots,B_n$ (which are parallel to the coordinate axes) and positive reals $\alpha_1,\dots,\alpha_n$ such that 
$\nu:=\sum \alpha _k \,\cU(B_k)$ satisfies
$$
 \nu\le \mu \text { and } \|\mu-\nu\|\le \eps.
$$
The term \textit{approachable from above} is defined analogously.
\end{Definition}

\begin{Remark}
Since we can express a measure which is approachable from below or above as the limit of a sequence of measures with Lebesgue density, 
it has itself a Lebesgue density. Conversely, any finite measure which has a density which is Riemann integrable on any cube, 
is approachable from below and above.
\end{Remark}

\begin{Proposition}\label{prop2912-1}
Let $\mu$ denote a compactly supported probability measure that is approachable from below. Further let $p\in\left[1,d/2\right)$. Then 
$$
\limsup_{N\to\infty} N^{1/d}  \, \E[\rho_p^p(\mu,\hat\mu_N)]^{1/p} \le \kappa_p^{\mathrm{unif}} \, \left(\int_{\R^{d}} 
\left(\frac{\dd \mu}{\dd \lambda^d}\right)^{1-\frac{p}{d}}\dd \lambda^d\right)^{1/p}.
$$
\end{Proposition}

\begin{proof} Let $\eps>0$ 
and choose a finite number of  pairwise disjoint cubes  $B_1,\dots,B_K$ and positive reals $\alpha_1,\dots,\alpha_K$ such that 
$\mu^*:=\sum_{k=1}^K \alpha _k \cU(B_k)\le \mu$ and  $\|\mu-\mu^*\|\le \eps$. For $k=1,\dots,K$ let  $\mu^{(k)}=\cU(B_k)$, 
set $\alpha_0 =\|\mu-\mu^*\|$ and fix a probability measure  $\mu^{(0)}$  such that
$$
\mu= \sum_{k=0}^K \alpha_k \mu^{(k)}.
$$
For each $k$, we consider empirical measures $(\hat\mu^{(k)}_n)_{n\in\N}$ of a sequence of independent $\mu^{(k)}$-distributed 
random variables. We assume independence of the individual empirical measures and observe that for an additional independent 
multinomial random variable $M=(M_k)_{k=0,\dots,K}$ with parameters $N$ and $(\alpha_k)_{k=0,\dots,K}$ one has
$$
N\,\hat \mu_N\stackrel{\mathcal{L}}= \sum_{k=0}^K M_k\, \hat \mu^{(k)}_{M_k}.
$$
We assume without loss of generality strict equality in the last equation.
Set $\nu=\sum_{k=0}^K \frac{M_k}{N} \mu^{(k)}$ and observe that by the triangle inequality
$$
\E[\rho_p^p (\mu,\hat\mu_N)]^{1/p} \le  \E[\rho_p^p (\mu,\nu)]^{1/p} + \E[\rho_p^p (\nu,\hat \mu_N)]^{1/p}.
$$
The first expression on the right hand side is of order $\cO(N^{-1/2p})$, (see proof of Proposition \ref{Pierce}). By 
Theorem \ref{thm2} (i), there is  a concave function $\varphi:[0,\infty)\to \R$ such that 
$\E [n\, \rho_p^p (\cU([0,1)^d),\widehat{\cU([0,1)^d)}_n)]\le \varphi(n)$ for all $n\in\N_0$ and
$$
\lim_{n\to\infty} \frac1{n^{1-p/d}} \,\varphi(n) = (\kappa_p^{\mathrm{unif}})^p.
$$
Denote by $a_1,...,a_K$ the edge lengths of the cubes $B_1,...,B_K$ and let $a_0>0$ be such that the support of $\mu$ is contained in 
a cube with side length $a_0$. Then, by Lemma \ref{wasser_prop} and Jensen's inequality,
\begin{align*}
N \,\E[\rho_p^p (\nu,\hat \mu_N)]&\le \sum_{k=0}^K \E[M_k\, \rho_p^p (\mu^{(k)},\hat \mu^{(k)}_{M_k})]\\
&\le (\kappa_p^{\mathrm{cube}})^p\,a_0^p \, \E [M_0^{1-p/d}]+ \sum_{k=1}^K a_k^p \,  \E [\varphi(M_k)]\\
&\le (\kappa_p^{\mathrm{cube}})^p\,a_0^p \, (\alpha_0N)^{1-p/d}+ \sum_{k=1}^K a_k^p \,  \varphi(\alpha_k N),
\end{align*}
so that
$$
\limsup_{N\to\infty} N^{p/d}\,\E[\rho_p^p (\nu,\hat \mu_N)]\le  (\kappa_p^{\mathrm{cube}})^p\,a_0^p \, \eps^{1-p/d} + 
(\kappa_p^{\mathrm{unif}})^p \sum_{k=1}^K a_k^p \,  \alpha_k^{1-p/d}.
$$
Note that for $x\in B_k$, $f(x):=\frac{\dd \mu_a}{\dd \lambda^d}\ge \alpha_k/a_k^d$ and we get
$$
 a_k^p \,  \alpha_k^{1-p/d} =\int_{B_k} a_k^{p-d} \alpha_k^{1-p/d}\, \dd x\le \int_{B_k}f(x)^{1-p/d}\, \dd x.
$$
Finally, we arrive at
$$
\limsup_{N\to\infty} N^{p/d}\,\E[\rho_p^p (\mu,\hat \mu_N)]\le  (\kappa_p^{\mathrm{unif}})^p \int_{\R^d} f(x)^{1-p/d}\,\dd x+  
(\kappa_p^{\mathrm{cube}})^p\,a_0^p \, \eps^{1-p/d}.
$$
Letting $\eps \to 0$ the assertion follows.
\end{proof}

\begin{Proposition}\label{prop2912-2}
Let $\mu$ be a finite singular measure on the Borel sets of $[0,1)^d$. For $p\in[1,d/2)$, one has
$$
\lim_{N\to\infty} N^{1/d} \,V^\mathrm{rand}_{N,p}(\mu) =0.
$$
\end{Proposition}

\begin{proof}
Without loss of generality we will assume that $\mu$ is a probability measure. 
Let $\eps>0$ and choose an open set $U\subset \R^d$ such that $\mu(U)=1$ and $\lambda^d(U)<\eps$. 
We fix finitely many pairwise disjoint cubes $B_1,\dots,B_K$ with
$$
U\supset B_1\cup\dots \cup B_K \ \text{ and } \mu(B_1\cup\dots \cup B_K)\ge 1-\eps.
$$
We set $B_0=[0,1)^d\backslash (B_1\cup\dots \cup B_K)$ and define the probability measure $\nu$, as in Lemma \ref{le0301-1}, 
by $\nu:=\sum_{k=0}^{K}\nu|_{B_k}$ where
$$
\nu|_{B_k}=\frac{\hat\mu_N(B_k)}{\mu(B_k)} \mu|_{B_k}.
$$
Then the vector $Z:=(N\hat \mu_N(B_k))_{k=0,\dots,K}$ is multinomially distributed with parameters $N$ and   $(\mu(B_k))_{k=0,\dots,K}$. 
Hence, by Lemma~\ref{le0301-1}, 
\begin{align}\label{eq2912-1}
\E [\rho_p^p(\mu,\nu)]^{1/p} \le \Bigl( \frac1{2N} \mathfrak{d^p} \sum_{k=0}^K \E |Z_k- N\mu(B_k)| \Bigr)^{1/p}=\cO(N^{-1/2p}).
\end{align}
We denote by $a_1,\dots, a_K$ the edge lengths of the cubes $B_k$, i.e. $a_k=\lambda^d(B_k)^{1/d}$, and set $a_0=1$. 
Note that $\nu|_{B_k}$ and $\hat\mu_N|_{B_k}$ have the same mass for all $k$. We apply Lemma \ref{wasser_prop}, 
Proposition \ref{coarse bound new} and Jensen's inequality to deduce that
\begin{align*}
 \E [\rho_p^p(\nu,\hat\mu_N)] &\le \sum_{k=0}^K \E [\rho_p^p(\nu|_{B_k},\hat\mu_N|_{B_k})] \le  \frac{1}{N}(\kappa_p^{\mathrm{cube}})^p 
\sum_{k=0}^K a_k^p \, \E\left[\left(\hat \mu_N({B_k})\, N\right)^{1-p/d}\right]\\
&\le (\kappa_p^{\mathrm{cube}})^p\,N^{-p/d}  \sum_{k=0}^K a_k^p \, (\mu({B_k}))^{1-p/d}.
\end{align*}
Next, we apply H\"older's inequality with exponents $d/p$ and $(1-p/d)^{-1}$ to get
\begin{align*}
\E [\rho_p^p(\nu,\hat\mu_N)] &\le 
 (\kappa_p^{\mathrm{cube}})^p \left({\sum_{k=1}^{K} \lambda^d(B_k)}\right)^{p/d} \cdot 
\left({\sum_{k=1}^{K}\mu(B_k)}\right)^{1-p/d} N^{-p/d} \\
& \ \ \ \ \ + (\kappa_p^{\mathrm{cube}})^p \, \mu(B_0)^{1-p/d}  N^{-p/d} \\
& \leq (\kappa_p^{\mathrm{cube}})^p (\eps^{p/d} + \eps^{1-p/d}) \, N^{-p/d}.
\end{align*}
It follows from (\ref{eq2912-1}) and the triangle inequality  that
$$
\limsup_{N\to\infty} N^{1/d} \, \E [\rho_p^p(\mu,\hat\mu_N)]^{1/p} \le \kappa_p^\mathrm{cube} (\eps^{p/d}+\eps^{1-p/d})^{1/p}
$$which finishes the proof since $\eps>0$ is arbitrary.
\end{proof}

\begin{Theorem}\label{asymptotic}
Let $p\in[1,\frac d2)$ and let $\mu$ denote a probability measure on $\R^d$ with finite $q$th moment for some $q>\frac{dp}{d-p}$. 
If the absolutely continuous part $\mu_a$ of $\mu$ has density $f$ which is approachable from below, then
\begin{equation}
\limsup\limits_{N \rightarrow \infty}N^{1/d} \, V_{N,p}^{\mathrm{rand}}(\mu) \leq \kappa_p^{\mathrm{unif}} \,  
\left(\int_{\R^{d}} f(x)^{1-\frac{p}{d}}\,\dd x \right)^{1/p}.
\end{equation}
If the absolutely continuous part $\mu_a$ of $\mu$ has density $f$ which is approachable from above, then
\begin{equation}
\liminf\limits_{N \rightarrow \infty}N^{1/d} \, V_{N,p}^\mathrm{rand}(\mu) \geq \underline{\kappa}_p^{\mathrm{unif}} \,  
\left(\int_{\R^{d}} f(x)^{1-\frac{p}{d}}\,\dd x \right)^{1/p}.
\end{equation}
\end{Theorem}

\begin{proof} We only prove the first statement since the second one is proved analogously (first establishing a corresponding 
version of Proposition~\ref{prop2912-1}). Let $\delta>0$ and set 
$$\mu^{(1)}=\frac{ \mu_a\big|_{B(0,\delta)}}{\mu_a(B(0,\delta))},  \ \mu^{(2)}= 
\frac{\mu_s\big|_{B(0,\delta)}}{\mu_s(B(0,\delta))} , \text{ and }  \mu^{(3)}= \frac{\mu\big|_{B(0,\delta)^c}}{\mu(B(0,\delta)^c)},
$$
where we let $\mu^{(i)}$ be an arbitrary probability measure in case the denominator is zero. 
As in the proof of Proposition~\ref{prop2912-1}, we represent $\hat \mu_N$ with the help of independent sequences of 
empirical measures 
$(\hat\mu^{(1)}_n)_{n\iN_0}, \dots, (\hat\mu^{(3)}_n)_{n\iN_0}$ and an independent  multinomially distributed random variable 
$M=(M_k)_{k=1,2,3}$ with parameters $N$ and $(\mu_a(B(0,\delta)),\mu_s(B(0,\delta)),\mu(B(0,\delta)^c))$ as
$$
N\hat\mu_N= \sum_{k=1}^3 M_k\, \hat\mu_{M_k}^{(k)}.
$$
As before one observes that for the random measure $\nu=\sum_{k=1}^3 \frac{M_k}{N} \mu^{(k)}$
$$
\E[\rho_p^p (\mu,\nu) ]^{1/p}=\mathcal{O}(N^{-1/2}).
$$
Further, by Lemma \ref{wasser_prop},
$$
N \,\E[\rho_p^p(\nu,\hat \mu_N)] \le  \sum_{k=1}^3 \E[ M_k\, \rho_p^p(\mu^{(k)}, \hat \mu^{(k)}_{M_k})]
$$
and,  by Propositions~\ref{prop2912-1} and \ref{prop2912-2}, there exist concave functions $\vphi_1$ and $\vphi_2$ with
$$
n \,V_{n,p}^\mathrm{rand}(\mu^{(k)})^p \le \vphi_k(n), \qquad \text{ for }n\in\N,\quad k=1,2
$$
and
$$
\vphi_1(n)\sim (\kappa_p^\mathrm{unif})^p \, n^{1-p/d} \int_{B(0,\delta)} \frac{f(x)^{1-p/d}}{ \mu_a(B(0,\delta))^{1-p/d} } \,\dd x  \ 
\text{ and } \ \vphi_2(n)=\mathrm{o}(n^{1-p/d})
$$
as $n\to\infty$. By Jensen's inequality,  $\E[ M_k\, \rho_p^p(\mu^{(k)}, \hat \mu^{(k)}_{M_k})]\le \vphi_k(\E[M_k])$ so that
$$
\limsup_{N\to\infty} \frac1{N^{1-p/d}}\, \E[M_1\, \rho_p^p(\mu^{(1)}, \hat \mu^{(1)}_{M_1})]\le  
(\kappa_p^\mathrm{unif})^p  \int_{B(0,\delta)} {f(x)^{1-p/d}} \,\dd x.
$$
Analogously, using Proposition~\ref{prop2912-2},
$$
\limsup_{N\to\infty} \frac1{N^{1-p/d}}\, \E[M_2 \,\rho_p^p(\mu^{(2)}, \hat \mu^{(2)}_{M_2})]=0
$$
and, by Theorem~\ref{Pierce},
$$
\limsup_{N\to\infty} \frac1{N^{1-p/d}}\, \E[M_3\, \rho_p^p(\mu^{(3)}, \hat \mu^{(3)}_{M_3})] \le 
(\kappa_{p,q}^\mathrm{Pierce})^p \Bigl[\int_{B(0,\delta)^c} \|x\|_\mathrm{max}^q\,\dd \mu(x)\Bigr]^{p/q} ,
$$
where we used that $1-\frac pd-\frac pq\geq 0$.
Altogether, we get
\begin{align*}
\limsup_{N\to\infty} &{N^{p/d}}\, \E[ \rho_p^p(\mu,\hat\mu_N)]\\
&\le   (\kappa_p^\mathrm{unif})^p  \int_{B(0,\delta)} {f(x)^{1-p/d}} \,\dd x+ (\kappa_{p,q}^\mathrm{Pierce})^p 
\Bigl[\int_{B(0,\delta)^c} \|x\|_\mathrm{max}^q\,\dd \mu(x)\Bigr]^{p/q}
\end{align*}
and letting $\delta\to\infty$ finishes the proof. 
\end{proof}

\subsection{Proof of the high resolution formula for  $p=1$}

In this section, we consider the special case $p=1$. We will write $\rho$ instead of $\rho_1$. The case $p=1$ is special because of the 
following lemma. 
\begin{Lemma}\label{Steffen}
Let $\mu,\,\nu,\,\kappa$ be finite measures on  $\R^d$ such that $\|\mu\|=\|\nu\|$. Then one has
$$
\rho(\mu+\kappa,\nu+\kappa)=\rho(\mu,\nu).
$$
\end{Lemma}
\begin{proof}
One has
\begin{align*}
\rho(\mu+\kappa,\nu+\kappa)&=\sup\{\int f \,\dd (\mu+\kappa) - \int f \,\dd (\nu+\kappa): f\;\mbox{1-Lipschitz}\}\\
&=\sup\{\int f \,\dd \mu - \int f \,\dd \nu: f\;\mbox{1-Lipschitz}\}=\rho(\mu,\nu).
\end{align*}
\end{proof}

The following lemma shows that the map $\mu \mapsto \limsup_{N\to\infty} \big( N^{1/d} V_{N,1}^\mathrm{rand}(\mu)\big)$ and likewise $\mu \mapsto \liminf_{N\to\infty} \big( N^{1/d} V_{N,1}^\mathrm{rand}(\mu)\big)$ are continuous with 
respect to the total variation norm.

\begin{Lemma}\label{total}
Let $d\ge 3$ and $q>\frac d{d-1}$. For probability measures $\mu$ and $\nu$ on $\R^d$ one has
$$
\limsup_{N\to\infty} N^{\frac1d} |V_{N,1}^\mathrm{rand}(\mu)-V_{N,1}^\mathrm{rand}(\nu)| \le 2 \kappa_{1,q}^\mathrm{Pierce} \|\mu-\nu\| ^{1-\frac1d-\frac 1q}
\Bigl(\int \|x\|^q_\mathrm{max}\,|\mu-\nu|(\dd x)\Bigr)^{\frac 1q}.
$$
\end{Lemma}
\begin{proof}
Without loss of generality, we assume that $\mu \neq \nu$.  
Let $\alpha= \frac {\mu\wedge \nu}{\|\mu\wedge \nu\|} $, $\mu^{*}=\frac{\mu- \mu\wedge \nu }{\|\mu- \mu\wedge \nu\|}$ and 
$\nu^*=\frac{\nu- \mu\wedge \nu }{\|\nu- \mu\wedge \nu\|}$ (let $\alpha$ be an arbitrary probability measure in case $\mu \wedge \nu =0$).
For fixed $N\in\N$ let $(M_1,M_2)$ be multinomially distributed with parameters $N$ and $(\|\mu\wedge \nu\|,1-\|\mu\wedge \nu\| )$.
We represent $\hat \mu_N$ and $\hat \nu_N$ as combinations of independent empirical measures $(\hat \alpha_n)$, $(\hat\mu^*_n)$ and 
$(\hat \nu^*_n)$ as
$$
N \,\hat\mu_N=M_1 \hat\alpha_{M_1}+M_2 \hat\mu^*_{M_2} \ \text{ and } \ N \,\hat\nu_N=M_1 \hat\alpha_{M_1}+M_2 \hat\nu^*_{M_2}.
$$
Then
\begin{align}\label{absch}
\begin{split}
\rho(N\mu, N \hat \mu_N) &\le \rho(N \mu, M_1 \alpha + M_2 \mu^*) +\rho(   M_1 \alpha + M_2 \mu^*, M_1 \hat \alpha_{M_1}+ M_2 \hat \mu^*_{M_2})\\
&\le  \rho(N \mu, M_1 \alpha + M_2 \mu^*) + \rho(   M_1 \alpha , M_1 \hat \alpha_{M_1})+\rho(  M_2 \mu^*,  M_2 \hat \mu^*_{M_2}).
\end{split}
\end{align}
Observe that 
\begin{equation}\label{wurzel}
\E[ \rho(N \mu, M_1 \alpha + M_2 \mu^*)]   = \cO(N^{1/2}).
\end{equation}
Further, by Theorem~\ref{Pierce} and Jensen's inequality, one has 
\begin{equation}\label{pierce2}
\E[\rho(  M_2 \mu^*,  M_2 \hat \mu^*_{M_2})] \le  \kappa_{1,q}^{\mathrm{Pierce}} \|\mu-\nu\| ^{1-\frac1d-\frac 1q} 
N^{1-\frac 1d}\Bigl( \int \|x\|_\mathrm{max}^q\,(\mu-\nu)_+(\dd x)\Bigr)^{\frac 1q} + \cO(N^{\frac 12}),
\end{equation}
where we used that $(\mu-\nu)_+=\|\mu-\nu\|\,\mu^*$.
Conversely, by Lemma \ref{Steffen} and Lemma~\ref{wasser_prop},
\begin{align*}
\rho(M_1 \alpha, M_1 \hat \alpha_{M_1})  &= \rho(M_1 \alpha + M_2 \hat \nu^*_{M_2}, M_1 \hat \alpha_{M_1}M_2 \hat \nu^*_{M_2})\\
 &= \rho(M_1 \alpha + M_2 \hat \nu^*_{M_2},N\hat \nu_N)\\
 &\le \rho(N\nu,N\hat\nu_N)+\rho(M_1 \alpha + M_2 \hat \nu^*_{M_2},N\nu)\\
 &=  \rho(N\nu,N\hat\nu_N)+\rho(M_1 \alpha + M_2 \hat \nu^*_{M_2}+M_2 \nu^*,N\nu +M_2 \nu^*)\\
 &\le \rho(N\nu,N\hat\nu_N)+\rho(M_2 \hat \nu^*_{M_2},M_2 \nu^*)+\rho(M_1 \alpha +M_2 \nu^*,N\nu).
\end{align*}
The expected values of the last two summands can be estimated like \eqref{pierce2} and \eqref{wurzel}. 
Inserting the estimates into \eqref{absch}, the assertion of 
the lemma follows. 
\end{proof}

We now prove the general upper and lower bounds in the case $p=1$.\\ 

\begin{proof}[ of Theorem \ref{thm2} (ii) for $p=1$]
Let $\mu=\mu_a+\mu_s$ be the Lebesgue decomposition of $\mu$ and let $f$ denote the density of $\mu_a$. 
It is now straightforward to verify that $\mu^{(n)}$ with density
$$
f^{(n)}(x) = 2^{-nd} \int_{S_{n,m_1,...,m_d}} f(y)\,\dd y \ \text{ for } x \in  S_{n,m_1,...,m_d},
$$
where $S_{n,m_1,...,m_d}:=2^{-n} ([m_1,m_1+1)\times \dots \times [m_d,m_d+1))$, 
satisfies $\|\mu_a- \mu^{(n)}\|\to 0$ and $ \int \|x\|_\mathrm{max}^q\,|\mu_a-\mu^{(n)}|(\dd x) \to 0$. Since $\mu^{(n)}+\mu_s$ is approachable from below and above,  
 Lemma \ref{total} allows to extend the upper and lower bounds of Theorem~\ref{asymptotic} to the case with general density if $p=1$. 
\end{proof}

\noindent{\bf Acknowledgement.} Reik Schottstedt acknowledges support from  DFG Grant SPP-1324 DE 1423/3-1.


\begin{thebibliography}{MGRY11}

\bibitem[AKT83]{AKT83}
M.~Ajtai, J.~Koml\'os, and G.~Tusn\`ady.
\newblock On optimal matchings.
\newblock {\em Combinatorica}, 4(4):259--264, 1983.

\bibitem[BW82]{BuWi82}
J.~A. Bucklew and G.~L. Wise.
\newblock {Multidimensional asymptotic quantization theory with rth power
  distortion measures.}
\newblock {\em IEEE Trans. Inf. Theory}, 28:239--247, 1982.

\bibitem[Coh04]{Coh04}
P.~Cohort.
\newblock {Limit theorems for random normalized distortion.}
\newblock {\em Ann. Appl. Probab.}, 14(1):118--143, 2004.

\bibitem[Der09]{Der09}
S.~Dereich.
\newblock {Asymptotic formulae for coding problems and intermediate
  optimization problems: a review.}
\newblock {In Trends in stochastic analysis. Cambridge University Press.
  187-232 }, 2009.

\bibitem[DGLP04]{DGLP04}
S.~Delattre, S.~Graf, H.~Luschgy, and G.~Pag\`es.
\newblock {Quantization of probability distributions under norm-based
  distortion measures.}
\newblock {\em Stat. Decis.}, 22(4):261--282, 2004.

\bibitem[DV11]{DerVor11}
S.~Dereich and C.~Vormoor.
\newblock The high resolution vector quantization problem with {O}rlicz norm
  distortion.
\newblock {\em J. Theor. Probab.}, 24(2):517--544, 2011.

\bibitem[GB11]{PW11}
Pag\`es G. and Wilbertz B.
\newblock "optimal delaunay and voronoi quantization schemes for pricing
  american style options".
\newblock Preprint, 2011.

\bibitem[GG92]{GG92}
A.~Gersho and R.~M. Gray.
\newblock {\em Vector Quantization and Signal Processing}, volume~1 of {\em The
  Springer International Series in Engineering and Computer Science}.
\newblock Springer, 1992.

\bibitem[GL00]{GL00}
S.~Graf and H.~Luschgy.
\newblock {\em {Foundations of Quantization for Probability Distributions.}}
\newblock {Lecture Notes in Mathematics. 1730. Berlin: Springer}, 2000.

\bibitem[HS10]{HS10}
M.~Huesmann and T.~Sturm.
\newblock {Optimal transport from {L}ebesgue to {P}oisson}.
\newblock http://arxiv.org/abs/1012.3845v1, 2010.

\bibitem[Kan42]{Kan42}
L.~V. Kantorovich.
\newblock On the translocation of masses.
\newblock {\em Doklady Akademii Nauk SSSR}, 37(No. 7-8):227--229, 1942.

\bibitem[KR58]{KR58}
L.~V. Kantorovich and G.~Rubinstein.
\newblock On a space of completely additive functions.
\newblock {\em Vestnik Leningrad Univ. Math.}, 13(7):52--59, 1958.

\bibitem[MGRY11]{MRY11}
T.~M\"uller-Gronbach, K.~Ritter, and L.~Yaroslavtseva.
\newblock {A derandomization of the {E}uler scheme for scalar stochastic
  differential equations}.
\newblock Preprint 80, DFG Priority Program 1324,
  http://www.dfg-spp1324.de/publications.php?lang=en, 2011.

\bibitem[Mon81]{Mon81}
G.~Monge.
\newblock M\'emoire sur la th\'eorie des d\'eblais et des remblais.
\newblock {\em M\'emoires de l'Acad\'emie Royale des Sciences},
  XVIII-XIX:666--704, 1781.

\bibitem[Pag98]{Pag98}
G.~Pag\`es.
\newblock {A space quantization method for numerical integration.}
\newblock {\em J. Comput. Appl. Math.}, 89(1):1--38, 1998.

\bibitem[PP03]{PagPrin03}
G.~Pag\`es and J.~Printems.
\newblock {Optimal quadratic quantization for numerics: the Gaussian case.}
\newblock {\em Monte Carlo Methods Appl.}, 9(2):135--165, 2003.

\bibitem[PPP03]{PPP03}
G.~Pag\`es, H.~Pham, and J.~Printems.
\newblock Optimal quantization methods and applications to numerical problems
  in finance.
\newblock {\em Handbook on Numerical Methods in Finance}, 2004:253--298, 2003.

\bibitem[PW10]{PagWil10}
G.~Pag\`es and B.~Wilbertz.
\newblock Sharp rate for the dual quantization problem.
\newblock Preprint, 2010.

\bibitem[RR98a]{RR98i}
S.~T. Rachev and L.~R\"uschendorf.
\newblock {\em Mass Transportation Problems}, volume~I of {\em Probability and
  its applications}.
\newblock Springer, 1998.

\bibitem[RR98b]{RR98ii}
S.~T. Rachev and L.~R\"uschendorf.
\newblock {\em Mass Transportation Problems}, volume~II of {\em Probability and
  its applications}.
\newblock Springer, 1998.

\bibitem[Tal94]{Tal94}
M.~Talagrand.
\newblock The transportation cost from the uniform measure to the empirical
  measure in dimension $\geq 3$.
\newblock {\em The Annals of Probability}, 22(2):919--959, 1994.

\bibitem[Was69]{Was69}
L.~N. Wasserstein.
\newblock Markov processes over denumerable products of spaces describing large
  systems of automata.
\newblock {\em Prob. Inf. Transmission}, 5:47--52, 1969.

\bibitem[Yuk08]{Yuk08}
J.~Yukich.
\newblock {Limit theorems for multi-dimensional random quantizers.}
\newblock {\em Electron. Commun. Probab.}, 13:507--517, 2008.

\bibitem[Zad66]{Zad66}
P.~L. Zador.
\newblock {Topics in the asymptotic quantization of continuous random
  variables}.
\newblock Bell Laboratories Technical Memorandum, 1966.

\end{thebibliography}

\end{document}